\documentclass[12pt]{article}
\usepackage[english]{babel}
\usepackage{latexsym,amssymb,amsmath,amsthm,amsfonts}
\usepackage{graphicx}
\usepackage{url}

\def\'#1{\ifx#1i{\accent"13 \i}\else{\accent"13 #1}\fi}

\newtheorem{theorem}{Theorem}[section]

\newtheorem{lemma}[theorem]{Lemma}

\begin{document}
\title{Achromatic numbers of Kneser graphs}

\author{M. G. Araujo-Pardo \footnotemark[1] \\ \url{garaujo@matem.unam.mx}
\and J. C. D{\' i}az-Pati{\~ n}o \footnotemark[1] \\ \url{juancdp@im.unam.mx}
\and C. Rubio-Montiel \footnotemark[2] \\ \url{christian.rubio@acatlan.unam.mx}}

\maketitle

\def\thefootnote{\fnsymbol{footnote}}
\footnotetext[1]{Instituto de Matem{\' a}ticas Unidad Juriquilla, Universidad Nacional Aut{\' o}noma de M{\' e}xico, 04510 Campus Juriquilla, Quer{\' e}taro, M{\' e}xico.}
\footnotetext[2]{Divisi{\' o}n de Matem{\' a}ticas e Ingenier{\' i}a, FES Acatl{\' a}n, Universidad Nacional Aut{\'o}noma de M{\' e}xico, Naucalpan, Mexico.}

\begin{abstract}
Complete vertex colorings have the property that any two color classes have at least an edge between them. Parameters such as the Grundy, achromatic and pseudoachromatic numbers come from complete colorings, with some additional requirement. In this paper, we estimate these numbers in the Kneser graph $K(n,k)$ for some values of $n$ and $k$. We give the exact value of the achromatic number of $K(n,2)$.
\end{abstract}



\section{Introduction}

Since the beginning of the study of colorings in graph theory, many interesting results have appeared in the literature, for instance, the chromatic number of Kneser graphs. Such graphs give an interesting relation between finite sets and graphs.

Let $V$ be the set of all $k$-subsets of $[n]:=\{1,2,\dots,n\}$, where $1\leq k\leq n/2$. The \emph{Kneser graph} $K(n,k)$ is the graph with vertex set $V$ such that two vertices are adjacent if and only if the corresponding subsets are disjoint. Lov{\' a}sz \cite{MR514625} proved that $\chi(K(n,k))=n-2(k-1)$ via the Borsuk-Ulam theorem, see Chapter 38 of \cite{MR3288091}.

Some results on the Kneser graphs and parameters of colorings have appeared since then, for instance \cite{rubio2016achromatic,MR2426534,MR2548791,MR2564784,MR1670155,MR2464882}.

An $l$-\emph{coloring} of a graph $G$ is a surjective function $\varsigma$ that assigns a number from the set $[l]$ to each vertex of $G$. An $l$-coloring of $G$ is \emph{proper} if  any two adjacent vertices have different colors. An $l$-coloring $\varsigma$ is \emph{complete} if for each pair of different colors $i,j\in [l]$ there exists an edge $xy\in E(G)$ such that $\varsigma(x)=i$ and $\varsigma(y)=j$. 

The largest value of $l$ for which $G$ has a complete $l$-coloring is called the \emph{pseudoachromatic number} of $G$ \cite{MR0256930}, denoted $\psi(G)$.  A similar invariant, which additionally requires an $l$-coloring to be proper, is called the \emph{achromatic number} of $G$ and denoted by $\alpha(G)$ \cite{MR0272662}. Note that $\alpha(G)$ is at least $\chi(G)$ since the \emph{chromatic number $\chi(G)$ of $G$} is the smallest number $l$ for which there exists a proper $l$-coloring of $G$ and then such an $l$-coloring is also complete. Therefore, for any graph $G$,  $\chi(G)\leq \alpha(G)\leq\psi(G).$

In this paper, we estimate these parameters arising from complete colorings of Kneser graphs. The paper is organized as follows. In Section \ref{Section2} we recall notions of block designs.

Section \ref{Section3} is devoted to the achromatic $\alpha(K(n,2))$ number of the Kneser graph $K(n,2)$. It is proved that $\alpha(K(n,2))=\left\lfloor \binom{n+1}{2}/3\right\rfloor$ for $n\not=3$.

In Section \ref{Section4} it is shown that $\psi(K(n,2))$ satisfies \[\left\lfloor\tbinom{n}{2}/2\right\rfloor\leq\psi(K(n,2))\leq\left\lfloor (\tbinom{n}{2}+\left\lfloor \frac{n}{2}\right\rfloor)/2\right\rfloor\] for $n\geq 7$ and that the upper bound is tight.

The Section \ref{Section5} establishes that the Grundy number $\Gamma(K(n,2))$ equals $\alpha(K(n,2))$. The \emph{Grundy number} $\Gamma(G)$ of a graph $G$ is determined by the worst-case result of a greedy proper coloring applied on $G$. A \emph{greedy} $l$-coloring technique operates as follows. The vertices (listed in some particular order) are colored according to the algorithm that assigns to a vertex under consideration the smallest available color. Therefore, greedy proper colorings are also complete.

Section \ref{Section6} gives a natural upper bound for the pseudoachromatic number of $K(n,k)$ and a lower bound for the achromatic number of $K(n,k)$ in terms of the $b$-chromatic number of $K(n,k)$, another parameter arising from complete colorings.

Section \ref{Section7} is about the achromatic numbers of some geometric type Kneser graphs. A \emph{complete geometric graph} of $n$ points is an embedding of the complete graph $K_n$ in the Euclidean plane such that its vertex set is a set $V$ of points in general position, and its edges are straight-line segments connecting pairs of points in $V$. We study the achromatic numbers of graphs $D_V(n)$ whose vertex set is the set of edges of a complete geometric graph of $n$ points and adjacency is defined in terms of geometric disjointness.

To end, in Section \ref{Section8}, we discuss the case of the odd graphs $K(2k+1,k)$.


\section{Preliminaries} \label{Section2}

All graphs in this paper are finite and simple. Note that the complement of the line graph of the complete graph on $n$ vertices is the Kneser graph $K(n,2)$. We use this model of the Kneser graph $K(n,2)$ in Sections \ref{Section3}, \ref{Section4}, \ref{Section5} and \ref{Section7}.

Let $n$, $b$, $k$, $r$ and $\lambda$ be positive integers with $n>1$. Let $D=(P,B,I)$ be a triple consisting of a set $P$ of $n$ distinct objects, called points of $D$, a set $B$ of $b$ distinct objects, called blocks of $D$ (with $P\cap B = \emptyset$), and an incidence relation $I$, a subset of $P\times B$. We say that $v$ is incident to $u$ if exactly one of the ordered pairs $(u,v)$ and $(v,u)$ is in $I$; then $v$ is incident to $u$ if and only if $u$ is incident to $v$.  $D$ is called a \emph{$2$-$(n, b, k, r, \lambda)$ block design} (for short, \emph{$2$-$(n, b, k, r, \lambda)$ design}) if it satisfies the following axioms.
\begin{enumerate}
\item Each block of $D$ is incident to exactly $k$ distinct points of $D$.
\item Each point of $D$ is incident to exactly $r$ distinct blocks of D.
\item If $u$ and $v$ are distinct points of $D$, then there are exactly $\lambda$ blocks of $D$ incident to both $u$ and $v$.
\end{enumerate}
A $2$-$(n, b, k, r, \lambda)$ design is called a balanced incomplete block design BIBD; it is called an $(n, k, \lambda)$-design, too, since the parameters of a $2$-$(n, b, k, r, \lambda)$ design are not all independent. The two basic equations connecting them are $nr=bk$ and $r(k-1)=\lambda(n-1)$. For a detailed introduction to block designs we refer to \cite{MR1729456,MR1742365}.

A design is \emph{resolvable} if its blocks can be partitioned into $r$ sets so that $b/r$ blocks of each part are point-disjoint and each part is called a \emph{parallel class}.

A \emph{Steiner triple system} $STS(n)$ is an $(n, 3, 1)$-design. It is well-known that an $STS(n)$ exists if and only if $n\equiv 1,3$ mod $6$. A resolvable $STS(n)$ is called a \emph{Kirkman triple system} and denoted by $KTS(n)$ and exists if and only if $n\equiv 3$ mod $6$, see \cite{MR0314644}.

An $(n, 5, 1)$-design exists if and only if $n\equiv 1,5$ mod $20$, see \cite{MR1742365}.

An $(n, k, 1)$-design can naturally be regarded as an edge partition into $K_k$ subgraphs, of the complete graph $K_{n}$.

Finally, we recall that the concepts of a 1-factor and a 1-factorization represent, for the case of $K_n$, a parallel class and a resolubility of an $(n, 2, 1)$-design, respectively.


\section{The exact value of $\alpha(K(n,2))$}\label{Section3}

In this section, we prove that $\alpha (K(n,2))=\left\lfloor  \binom{n+1}{2}/3\right\rfloor$ for every $n\not =3$. The proof is about the upper bound and the lower bound have the same value.

\begin{theorem} \label{teo1}
The achromatic number $\alpha(K(n,2))$ of $K(n,2)$ equals $\left\lfloor \binom{n+1}{2}/3\right\rfloor$ for $n\not=3$ and $\alpha(K(3,2))=1$.
\end{theorem}
\begin{proof}
First, we prove the upper bound $\alpha(K(n,2))\leq \left\lfloor  \binom{n+1}{2}/3\right\rfloor$.

Let $\varsigma$ be a proper and complete coloring of $K(n,2)$. Consider the graph $K(n,2)$ as the complement of $L(K_n)$.  Note that vertices corresponding to a color class of $\varsigma$ of size two induce a $P_3$ subgraph, say $abc$, of the complete graph $K_n$ with $V(K_n) = [n]$; then no color class of $\varsigma$ of size one is a pair containing $b$. Therefore, if $\varsigma$ has $x$ color classes of size one (they form a matching in $K_n$ of size $x$) and $y$ color classes of size two, then $y\leq n-2x$,
 \[\alpha(K(n,2))\leq  \frac{\binom{n}{2}-x-2(n-2x)}{3}+x+(n-2x) = \frac{\binom{n}{2}+2x+(n-2x)}{3}  =  \frac{\binom{n}{2}+n}{3}\]
and we get $\alpha(K(n,2))\leq \left\lfloor  \binom{n+1}{2}/3\right\rfloor$. For the case of $n=3$, $K(3,2)$ is an edgeless graph, hence $\alpha(K(3,2))= 1$. 

Next, we exhibit a proper and complete edge coloring of the complement of $L(K_n)$ that uses $\left\lfloor  \binom{n+1}{2}/3\right\rfloor$ colors. We remark that in order to obtain such a tight coloring it suffices to achieve that all color classes are of size at most three, while the number of \emph{exceptional} vertices of $K_n$ (that are involved neither in a color class of size one nor in the role of the ``center'' of a color class of size two) is at most one. We shall refer to this condition as the condition ($C$).

Figure \ref{Fig1} documents the equality for $n\leq 5$.
\begin{figure}[!htbp]
\begin{center}
\includegraphics[scale=0.7]{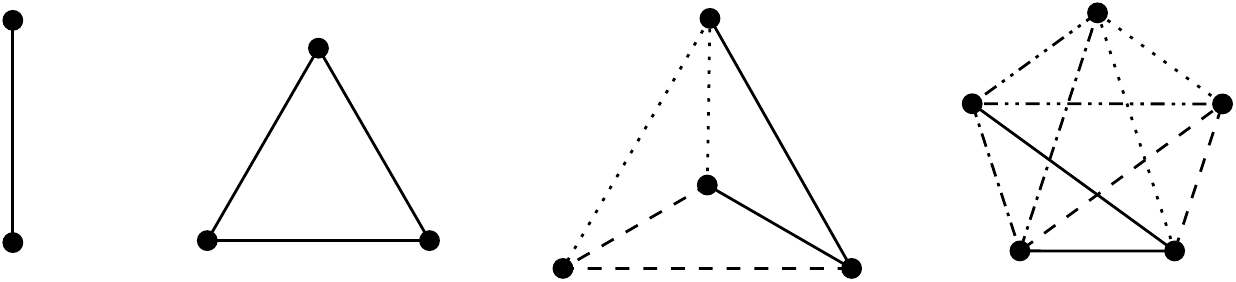}
\caption{\label{Fig1} $\alpha(K(n,2))=\left\lfloor \binom{n+1}{2}/3\right\rfloor$ for $n=2,4,5$ and $\alpha(K(3,2))=1$.}
\end{center}
\end{figure}
For the remainder of this proof, we need to distinguish four cases, namely, when $n=6k,6k+2$; $n=6k+3,6k+5$; $n=6k+4$ and $n=6k+1$ for $k\geq 1$. 
\begin{enumerate}
\item Case $n=6k$ or $n=6k+2$. Since $n+1\equiv 1,3$ mod $6$ there exists an  $STS(n+1)$. We can think of $K(n,2)$ as having the vertex set equal to the set of points of $STS(n+1)$ other than $v$. Then each vertex of $K(n, 2)$ is a subset of exactly one block of $STS(n+1)-v$; the blocks of $STS(n+1)-v$ are ($3$-element) blocks of $STS(n+1)$ not containing $v$, and ($2$-element) blocks $B\setminus \{v\}$, where $B$ is a block of $STS(n+1)$ with $v\in B$. Consider a vertex coloring of $K(n,2)$ that is defined in the following way: Color classes of size three are triangles of $STS(n+1)-v$ (we use this simplified expression to indicate that all vertices of $K(n,2)$, that are subsets of a fixed triangle of $STS(n+1)-v$, receive the same color). All remaining color classes are of size one; they are formed by $2$-element blocks of $STS(n+1)-v$. (They can also be regarded as edges of a perfect matching of the ``underlying'' complete graph on points of $STS(n+1)-v$.) The coloring is obviously proper.  It is complete, too (see Figure \ref{Fig2}), and satisfies the condition ($C$), hence $\alpha(K(n,2))=\left\lfloor \binom{n+1}{2}/3\right\rfloor$.
\begin{figure}[!htbp]
\begin{center}
\includegraphics[scale=0.7]{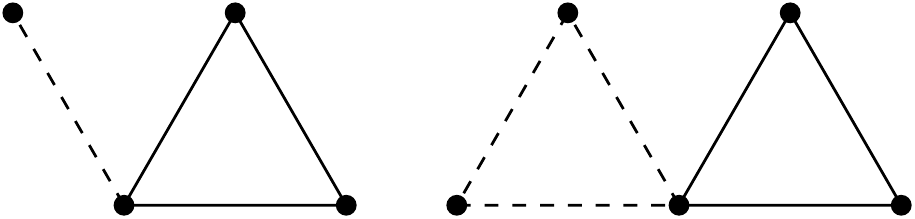}
\caption{\label{Fig2} Every two classes have two disjoint edges in the complement of $L(K_n)$.}
\end{center}
\end{figure}
\item Case $n=6k+3$ or $n=6k+5$.  Add two points $u$,  $v$ to the points of $STS(n-2)$, and color $K(n, 2)$ as follows. Color classes of size three are triangles of $STS(n-2)$ except for one with points $a, b, c$. The remaining color classes are of size two. Five of them correspond to the optimum coloring of $K(5,2)$ depicted in Figure \ref{Fig1} (with points $a,b,c,u,v$).  Finally, every point $x$ of $STS(n-2)$, $x\notin \{a, b, c\}$, gives rise to the color class $\{ux, xv\}$, see Figure \ref{Fig3} (Left). The coloring is proper and complete, see Figure \ref{Fig3} (Right), and it satisfies the condition ($C$).
\begin{figure}[!htbp]
\begin{center}
\includegraphics[scale=0.7]{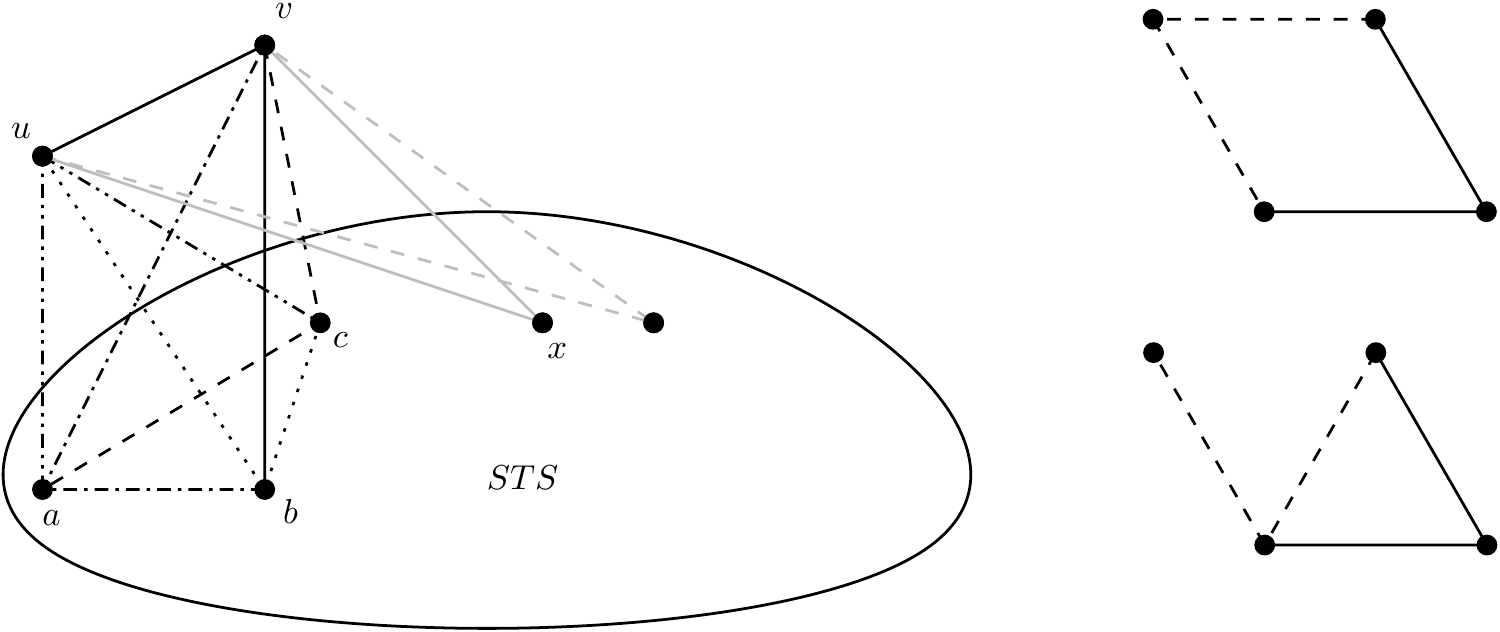}
\caption{\label{Fig3} (Left) Color classes of size two in the proof of Case 2. (Right) Every two distinct color classes of size two contain disjoint edges in the complement of $L(K_n)$.}
\end{center}
\end{figure}
\item Case $n=6k+4$. 
Add a point $v$ to the points of a resolvable $STS(n-1)$, for instance a $KTS(n-1)$. Color classes of size three are triangles of $STS(n-1)$ except for the triangles of a parallel class $P=\{T_i \colon i = 0,1,\dots, \frac{n-4}{3}\}$, where $T_i = \{v_{3i+1},v_{3i+2},v_{3i+3}\}$.  For each triangle $T_i$ of $P$ color vertices of $K(4,2)$ with the vertex set $T_i\cup  \{v\}$ according to the optimum coloring of Figure \ref{Fig1}, see Figure \ref{Fig4}. The resulting coloring is proper and complete, and it fulfills the condition ($C$).  Indeed, the number of color classes of size two is $n-1$; since ``centers'' of those color classes are pairwise disjoint, $v$ is the only exceptional vertex.
\begin{figure}[!htbp]
\begin{center}
\includegraphics[scale=0.7]{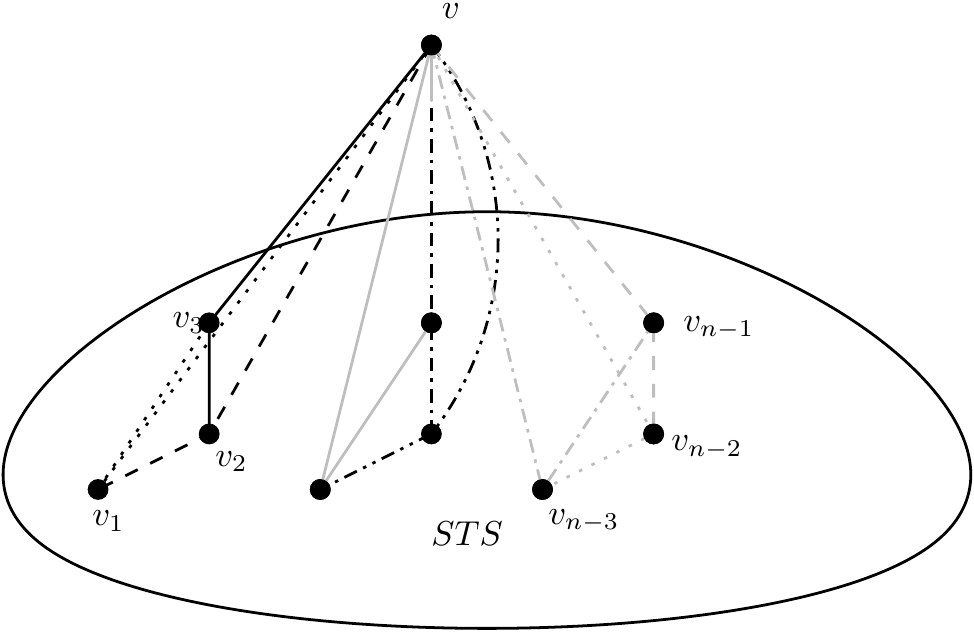}
\caption{\label{Fig4} Color classes of size two in the proof of Case 3.}
\end{center}
\end{figure}
\item Case $n=6k+1$.  First, we analyze the case of $k=1$.  Delete two points of $STS(9)$ presented in Figure \ref{Fig5} (Left) to finish with points $v_1, v_2,\dots , v_7$. The ``survived'' triangles are color classes of size three, see Figure \ref{Fig5} (Center). The remaining six pairs of points are divided into four color classes $\{v_1v_2,v_2v_3\}$, $\{v_3v_4\}$, $\{v_4v_5,v_5v_6\}$ and $\{v_6v_1\}$, see Figure \ref{Fig5} (Right). The obtained coloring shows that $\alpha(K(7,2))=9= \left\lfloor  \binom{7+1}{2}/3\right\rfloor$.
\begin{figure}[!htbp]
\begin{center}
\includegraphics[scale=0.8]{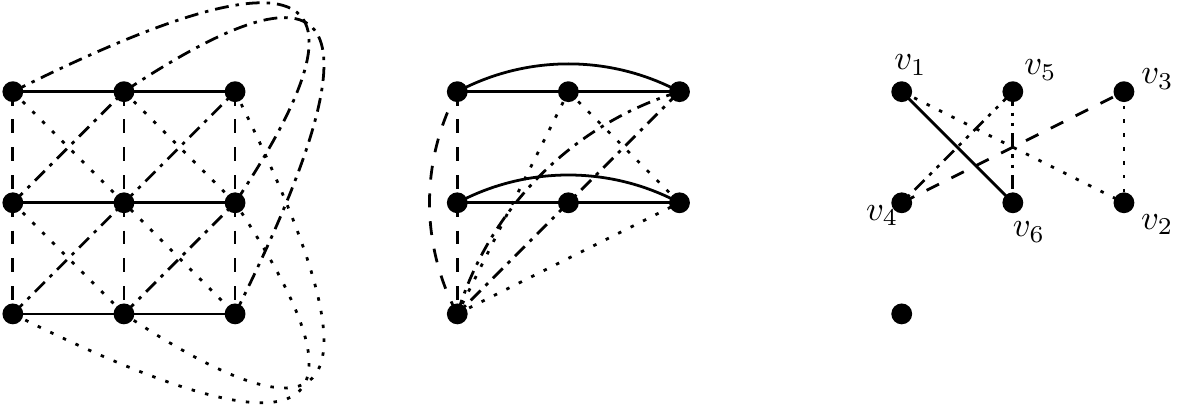}
\caption{\label{Fig5} (Left) $STS(9)$. (Center) The 5 color classes of size three of $K_7$. (Right) Color classes of size one and two in $K(7, 2)$.}
\end{center}
\end{figure}

If $k \geq 2$, consider an $STS(n-4)$ with points $v_1,v_2,\dots ,v_{n-4}$ that has a parallel class $P=\{T_i\colon  i = 0,1,\dots,\frac{n-7}{3}\}$, where $T_i = \{v_{3i+1},v_{3i+2},v_{3i+3}\}$. Add to points of $STS(n-4)$ the points $a, b, c, d$. Every triangle of $STS(n-4)$ except for the triangles of $P$ is a color class of size three. Let $H_i$ denote the join of $T_i$ with the complement of $K_4$ on vertices $a, b, c, d$; the join of two vertex disjoint graphs $G$ and $H$ has the vertex set $V(G)\cup V(H)$ and the edge set $E(G)\cup E(H) \cup \{xy\colon x\in V(G),y\in V(H)\}$. Pairs of points corresponding to edges of $H_i$, $i = 0, 1, \dots , \frac{n-10}{3}$, form color classes of size three determined by point triples $\{v_{3i+1},v_{3i+2},a\}$, $\{v_{3i+2},v_{3i+3},b\}$ and $\{v_{3i+1},v_{3i+3},c\}$, and color classe of size two $\{av_{3i+3}, v_{3i+3}d\}$, $\{bv_{3i+1}, v_{3i+1}d\}$ and $\{cv_{3i+2}, v_{3i+2}d\}$, see Figure \ref{Fig6}. Finally, pairs of points from the set $S = \{v_{n-6}, v_{n-5}, v_{n-4}, a, b, c, d\}$ are colored so that nine color classes are created just as in the coloring of $K(7,2)$ described above for the case $k = 1$. The coloring is proper and complete, and the condition ($C$) is fulfilled, since the number of exceptional vertices in the ``underlying'' $K_n$ is one (exceptional is the vertex of $S$ that is involved only in color classes of size three); so, $\alpha(K(n,2))=\left\lfloor \binom{n+1}{2}/3\right\rfloor$ in this case, too.
\begin{figure}[!htbp]
\begin{center}
\includegraphics[scale=0.8]{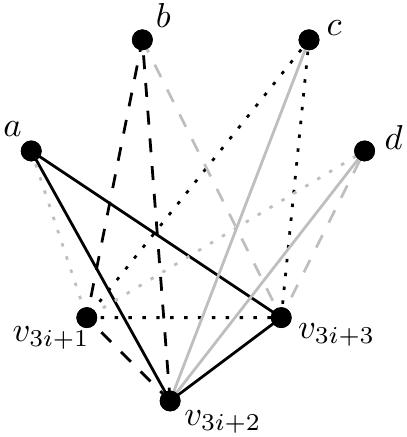}
\caption{\label{Fig6} The 6-coloring of $H_i$.}
\end{center}
\end{figure}
\end{enumerate}
By the four cases, the theorem follows.
\end{proof}


\section{About the value of $\psi(K(n,2))$}\label{Section4}

In this section, we determine bounds for $\psi (K(n,2))$. The gap between the bounds is  $\Theta(n)$, however, the upper bound is tight for an infinite number of values of $n$.

\begin{theorem} \label{teo2}
$\psi(K(n,2))=\alpha(K(n,2))$ for $2\leq n \leq 6$ and 
\[\left\lfloor \frac{\tbinom{n}{2}}{2}\right\rfloor\leq\psi(K(n,2))\leq\left\lfloor \frac{\tbinom{n}{2}+\left\lfloor \frac{n}{2}\right\rfloor}{2}\right\rfloor\]
for $n\geq 7$. Moreover, the upper bound is tight if $n+1\equiv 0,1$ mod $20$.
\end{theorem}
\begin{proof}
For $n=2,3$,  the graph $K(n,2)$ is edgeless and then $\psi(K(n,2))=\alpha(K(n,2))=1$.

For $n=4,5,6$, $\alpha(K(n,2))$ is $3,5,7$, respectively (by Theorem \ref{teo1}). Note that any complete coloring having
a color class of size one uses at most $k=2,4,7$ colors, respectively. And any complete coloring without color classes of size one uses at most $k=3,5,7$ colors, respectively. Therefore, $\psi(K(n,2))$ is at most $3,5,7$, respectively. Hence $\psi(K(n,2))=\alpha(K(n,2))$.

For $n\geq 7$, any complete coloring of $K(n,2)$ has at most $\omega(K(n,2))=\left\lfloor \frac{n}{2}\right\rfloor$ classes of size 1 ($\omega(G)$ is the clique number of the garph $G$, that is, the largest order of a complete subgraph of $G$), then \[\psi(K(n,2))\leq \left\lfloor \frac{\tbinom{n}{2}-\left\lfloor \frac{n}{2}\right\rfloor}{2}+\left\lfloor \frac{n}{2}\right\rfloor\right\rfloor=\left\lfloor \frac{\tbinom{n}{2}+\left\lfloor \frac{n}{2}\right\rfloor}{2}\right\rfloor.\]
Such an upper bound is proved in \cite{MR3461960}.

To see the lower bound, we use a 1-factorization $F$ of $K_{2t}$ such that no component induced by two distinct 1-factors
of $F$ is a 4-cycle, see \cite{MR867658,MR1022464}. We need to distinguish four cases, namely, when $n=4k-1$, $4k$, $n=4k+1$ and $n=4k+2$ for $k\geq 1$. 
\begin{enumerate}
\item Case $n=4k$. Consider $F$ for $t=2k$. Since each 1-factor contains $t$ edges, we have $k$ color classes of size two for each 1-factor, therefore the lower bound follows.
\item Case $n=4k+1$. Consider $F$ for $t=2k+1$ and delete a vertex of $K_{4k+2}$. Since each maximal matching arising from a 1-factor of $F$ contains $t-1$ edges, we have $k$ color classes of size two for each such maximal matching, hence the lower bound follows.
\item Case $n=4k+2$. Consider $F$ for $t=2k$ and add two new vertices $a$ and $b$ to $V(K_{4k})$ to obtain $K_{4k+2}$. Color the subgraph $K_{4k}$ as above, and the remaining edges as follows. For each vertex $x$ of $K_{4k}$, we have the classes $\{ax,xb\}$. Finally, color the edge $ab$ in a greedy way and the result follows.
\item Case $n=4k-1$. Consider $F$ for $t=2k-1$ and adding a new vertex $b$ to obtain $K_{4k-1}$. Color the subgraph $K_{4k-3} =K_{4k-1}-\{a,b\}$ as in the case $n \equiv 1$ mod $4$, and form for each vertex $x$ of $K_{4k-3}$ the color class $\{ax, xb\}$. Finally, choose for the edge $ab$ greedily a color that is already used; the result then follows.
\end{enumerate}
Now, to verify that the upper bound is tight, consider an $(n+1,5,1)$-design $D$, see \cite{MR1742365}. Therefore $n+1\equiv 1,5$ mod $20$.  Choose a point $v$ of $D$ and let $C=\{Q_i\colon i = 0,1,\dots, \frac{n-4}{4}\}$ with $Q_i = \{v_{4i+1},v_{4i+2},v_{4i+3},v_{4i+4}\}$ be the set of $4$-blocks of $D-v$. Pairs of points of every $5$-block of $D-v$ are colored so that five color classes of size two are created, see Figure \ref{Fig7}; the coloring is not proper, since all those color classes induce a $K_2$ subgraph of $K(n, 2)$.
\begin{figure}[!htbp]
\begin{center}
\includegraphics[scale=0.8]{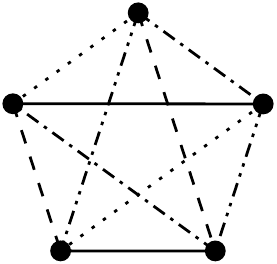}
\caption{\label{Fig7} A complete coloring of $K(5,2)$ using five colors that is not proper.}
\end{center}
\end{figure}

Label the edges of each block $Q_i$ of $C$ as $f_{2i} = v_{4i+1}v_{4i+2},$ $f_{2i+1} = v_{4i+3}v_{4i+4}$, $e_i$, $e_{n/4+i}$, $e_{n/2+i}$ and $e_{3n/4+i}$. Remaining vertices of $K(n,2)$ are colored to form color classes $\{f_i\}$ of size one for $i = 0,1,\dots,n/2-1$, and color classes $\{e_{2i},e_{2i+1}\}$ of size two for $i = 0, 1, .\dots , n/4-1$. The coloring is complete, hence the result follows.
\end{proof}


\section{On the Grundy number of $K(n,2)$}\label{Section5}

In this section, we observe that the coloring used in Theorem \ref{teo1} is also a greedy coloring. 

An $l$-coloring of $G$ is called \emph{Grundy}, if it is a proper coloring having the property that for every two colors $i$ and $j$ with $i<j$, every vertex colored $j$ has a neighbor colored $i$ (consequently, every Grundy coloring is a complete coloring).
Moreover, a coloring $\varsigma$ of a graph $G$ is a Grundy coloring of $G$ if and only if $\varsigma$ is a greedy coloring of $G$, see \cite{MR2450569}. Therefore, the Grundy number  $\Gamma(G)$ is the largest $l$ for which a Grundy $l$-coloring of $G$ exists.  Any graph $G$ satisfies, $\chi(G)\leq \Gamma(G)\leq \alpha(G)\leq \psi(G)$.

Consider the coloring used in Theorem \ref{teo1}.  Divide colors into \emph{small}, \emph{medium} and \emph{high} (recall that colors used in Theorem \ref{teo1} are positive integers), and use them for color classes of size three, two and one, respectively. We only need to verify that if $i$ and $j$ are colors with $i< j$, then for every edge $e$ of color $j$ there exists an edge of color $i$ that is disjoint with $e$. This is certainly true if $j$ is a high color, since the coloring is complete. If the color $j$ is not high, the required condition is satisfied because of the following facts: ($i$) ($3$-element) vertex sets corresponding to color classes $i$ and $j$ have at most one vertex in common; ($ii$) the centers of involved $P_3$ subgraphs are distinct if both $i$ and $j$ are medium colors.

Consider the coloring used in Theorems \ref{teo1}. Taking the highest colors as the color class of size 1 and the smallest colors as the color classes of size 3. We only need to verify that for every two color classes with colors $i$ and $j$, $i<j$, and every edge of color $j$ there always exist a disjoint edge of color $i$. This is true if the color classes are triangles because they only share at most one vertex. If the color classes are an triangle $K_3$ with color $i$ and a path $P_3$ with color $j$ this is also true.

\begin{theorem} \label{teo3}
$\Gamma(K(n,2))=\left\lfloor \binom{n+1}{2}/3\right\rfloor$ for $n\not=3$ and $\Gamma(K(3,2))=1$.
\end{theorem}


\section{About general upper bounds}\label{Section6}

The known upper bound for the pseudoachromatic number states, for $K(n,k)$ (see \cite{MR2450569}), that
\begin{equation}\label{eq1}
\psi(K(n,k))\leq \frac{1}{2}+\sqrt{\frac{1}{4}+\binom{n}{k}\binom{n-k}{k}}=O\left(\frac{n^{k/2}(n-k)^{k/2}}{k!}\right).
\end{equation}

A slightly improved upper bound is the following. Let $\varsigma$ be a complete coloring of $K(n,k)$ using $l$ colors with $l=\psi(K(n,k))$. Let $x= \min\{ \left|\varsigma^{-1}(i)\right| : i\in \left[l\right]\}$, that is, $x$ is the cardinality of the smallest color class of $\varsigma$; without loss of generality we may suppose that $x= \left|\varsigma^{-1}(l)\right|$. Since $\varsigma$ defines a partition of the vertex set of $K(n, k)$ it follows that $l\leq f(x):=\binom{n}{k}/x$. 

Additionally, since $K(n,k)$ is $\binom{n-k}{k}$-regular, there are at most $\binom{n-k}{k}$ vertices adjacent in $K(n, k)$ to a vertex of $\varsigma^{-1}(l)$. With $\mathcal{X}:=\bigcup_{X\in \varsigma^{-1}(l)}X$ we have $|[n]\setminus \mathcal{X}| \geq n-kx$.  If $n-kx\geq k$ and $Y \subseteq [n]\setminus \mathcal{X}$,  $|Y|=k$, then each of $x$ edges $XY$,  $X\in \varsigma^{-1}(l)$, corresponds to the pair of colors $l$, $\varsigma (Y)$. Therefore, $\psi (K(n, k)) \leq g(x)$, where $g(x):=1+x\binom{n-k}{k}-(x-1)\binom{n-kx}{k}$, if $n-kx\geq k$, and $g(x):=1+x\binom{n-k}{k}$ otherwise. Consequently, we have:  \[\psi(K(n,k))\leq\min\{f(x),g(x)\}.\]
Hence, we conclude that:
\[\psi(K(n,k))\leq\left\lfloor \max\left\{ \min\{f(x),g(x)\}\colon x\in\mathbb{N}\right\}\right\rfloor\]
and then
\[\psi(K(n,k))\leq\left\lfloor \max\left\{ \min\{f(x),g(x)\}\colon x\in\mathbb{R}^+\right\}\right\rfloor.\]
It is not hard to see that $\max\left\{ \min\{f(x),g(x)\}\colon x\in\mathbb{R}^+\right\}=\frac{1}{2}+\sqrt{\frac{1}{4}+\binom{n}{k}\binom{n-k}{k}}$ if $n-kx< k$.

On a general lower bound. An $l$-coloring $\varsigma$ is called \emph{dominating} if every color class contains a vertex that has a neighbor in every other color class. The \emph{b-chromatic number $\varphi(G)$ of $G$} is defined as the largest number $l$ for which there exists a dominating $l$-coloring of $G$ (see \cite{MR1670155}). Since a dominating coloring is also complete, hence, for any graph $G$, $\varphi(G) \leq \alpha(G).$ The following theorem was proved in \cite{MR2564784}:
\begin{theorem}[Hajiabolhassan \cite{MR2564784}]
Let $k\geq3$ an integer. If $n\geq 2k$, then $2 \binom{ \left\lfloor \frac{n}{2} \right\rfloor}{k} \leq \varphi(K(n,k)).$
\end{theorem}
In consequence, for any $n$, $k$ satisfying $n\geq 2k\geq 6$, we have 
\[\alpha(K(n,k)) \geq 2 \binom{\left\lfloor \frac{n}{2} \right\rfloor}{k} = \Omega\left( \frac{n^k}{2^{k-1}k^k}\right).\]


\section{The achromatic numbers of $D_V(n)$}\label{Section7}

Let $V$ be a set of $n$ points in general position in the plane, i.e., no three points of $V$ are collinear. The segment disjointness graph $D_V(n)$ has the vertex set equal to the set of all straight line segments with endpoints in $V$, and two segments are adjacent in $D_V(n)$ if and only if they are disjoint. Each graph $D_V(n)$ is a spanning subgraph of $K(n,2)$. The chromatic number of the graph $D_V(n)$ is bounded in \cite{MR2155418} where it is proved that $\chi(D_V(n))=\Theta(n).$

In this subsection, we prove bounds for $\alpha(D_V(n))$ and $\psi(D_V(n))$.  Having in mind the fact that $\psi(H) \leq \psi(G)$ if $H$ is a subgraph of $G$, Theorem \ref{teo2} yields
\[\psi(D_V(n))\leq \left\lfloor \frac{\binom{n}{2}+\left\lfloor\frac{n}{2}\right\rfloor}{2} \right\rfloor\leq \frac{n^2}{4}.\]

For the lower bound, we use the following results  A \emph{straight line thrackle} is a set $S$ of straight line segments such that any two distinct segments of $S$ either meet at a common endpoint or they cross each other (see \cite{MR15796}). 

\begin{theorem}[Erd{\H o}s \cite{MR15796} (see also the proof of Theorem 1 of \cite{MR2173933})]\label{teoErd}
If $d_1(n)$ denote the maximum number of edges of a straight line thrackle of $n$ vertices then $d_1(n)=n$.
\end{theorem}

\begin{lemma}\label{triangles}
Any two triangles $T_1$ and $T_2$ with points in $V$, that share at most one point, contain two disjoint edges.
\end{lemma}
\begin{proof}
Case 1. $T_1$ has a point in common with $T_2$: Since $T_1\cup T_2$ have five points and six edges, then two of its edges are disjoint due to $d_1(5)=5$.

Case 2. $T_1$ has no points in common with $T_2$: Let $e$ be an edge of $T_2$.  Let suppose that $T_1\cup T_2$ does not contain two disjoint edges, then $T_1$ and $e$ is a straight line thrackle.  Therefore, a vertex of $e$ and a vertex of $T_1$ have to be the same, which is impossible because $T_1$ has no points in common with $T_2$.
\end{proof}

Now, if we identify a Steiner triple system $STS(n)$ with the complete geometric graph of $n$ points and we color each triangle with a different color, by Lemma \ref{triangles}, we have the following.

\begin{lemma}
If $n\equiv 1,3$ mod $6$ and $V$ is a set of $n$ points in general position, then\[\frac{n^2}{6}-\frac{n}{6}= \frac{1}{3} \binom{n}{2}\leq \alpha(D_V(n))\]
\end{lemma}

Therefore, we have the following theorem.

\begin{theorem}
For any natural number $n$ and any set of $n$ points $V$ in general position,
\[\frac{n^2}{6}-\Theta(n)\leq \alpha(D_V(n)).\]
\end{theorem}

Further, if $K_n$ has an even number of vertices, then there is a set $F\subseteq E(K_n)$ such that $E(K_n)\setminus F$ can be partitioned into triangles. More precisely, if $n\equiv 0, 2$ mod $6$,then $F$ is a perfect matching in $K_n$, and if $n\equiv 4$ mod $6$, then $F$ induces a spanning forest of $n/2 + 1$ edges in $K_n$ with all vertices having an odd degree, see \cite{MR2071903,MR0382030}. 

A set $V$ of $n$ points in \emph{convex position} is a set of $n$ points in general position such that they are the vertices of a convex polygon (each internal angle is strictly less than 180 degrees).

\begin{theorem}\label{teo7.4}
For any even natural number $n$ and any set of $n$ points $V$ in convex position,
\[\frac{n^2}{6}+\Theta(n)\leq \alpha(D_V(n))\]
\end{theorem}
\begin{proof}
Take the edges of $F$ in the convex hull of $V$, except for one in the case of $n\equiv 4$ mod $6$, see Figure \ref{Fig8}. Each component of $F$ is a color class. Each triangle of $K_n-F$ is a color class. Essentially we use $\binom{n+1}{2}/3$, and the result follows.
\end{proof}

\begin{figure}[!htbp]
\begin{center}
\includegraphics[scale=0.8]{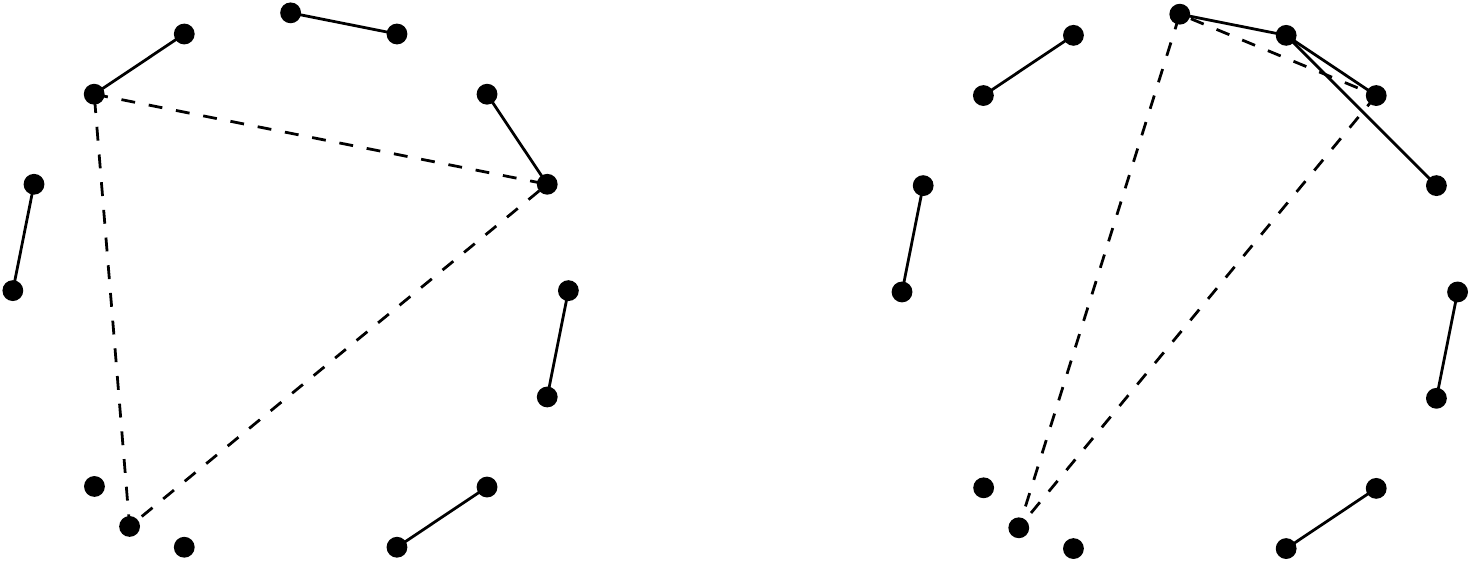}
\caption{\label{Fig8} Configurations of the color classes of size one arising from $F$ and a dashed triangle of $K_n-F$ in the proof of Theorem \ref{teo7.4}.}
\end{center}
\end{figure}

Finally, the geometric type Kneser graph $D_V(n,k)$ for $k\geq 2$ whose vertex set consists of all subsets of $k$ points in $V$. Two such sets $X$ and $Y$ are adjacent if and only if their convex hulls are disjoint. Given a point set $V$, for a line dividing $V$ into two sets $V_1$ and $V_2$ of $n/2$ points, having a coloring such that each color class has sets $X\subseteq V_1$ and $Y\subseteq V_2$,  we have that 
\[\psi(D_V(n,k)) \geq \binom{n/2}{k} = \Omega\left( \frac{n^k}{2^{k}k^k}\right)\]


\section{On odd graphs}\label{Section8}

It is obvious to prove that the achromatic and the pseudoachromatic number as well of (the graph induced by) a matching of size $\binom{k}{2}$ is equal to $k$. Therefore, a matching of size $m$ has achromatic and pseudoachromatic number equal to  $\left\lfloor \frac{1}{2}+\sqrt{\frac{1}{4}+2m} \right\rfloor$, which means that in the case $n = 2k$ the upper bound of (1) for $\psi (K(2k,k))$ is equal to the lower bound for $\alpha (K(2k,k))$; in other words,
\[\alpha(K(2k,k))=\psi(K(2k,k))=\left\lfloor \frac{1}{2}+\sqrt{\frac{1}{4}+\binom{2k}{k}}\right\rfloor\]

However, the situation is different in the case of $K(2k+1,k)$, the Kneser graphs that are called \emph{odd graphs}. The better lower bound we have is
\[\Omega \left( 2^{k/2} \right)=\left\lfloor \frac{1}{2}+\sqrt{\frac{1}{4}+\binom{2k}{k}}\right\rfloor\leq\psi(K(2k+1,k)),\]

due to the fact that $K(2k,k)$ is a subgraph of the odd graph $K(2k+1,k)$.


\section*{Acknowledgments}

The authors wish to thank the anonymous referees of this paper for their suggestions and remarks.

Part of the work was done during the I Taller de Matem{\' a}ticas Discretas, held at Campus-Juriquilla, Universidad Nacional Aut{\' o}noma de M{\' e}xico, Quer{\' e}taro City, Mexico on July 28-31, 2014. Part of the results of this paper was announced at Discrete Mathematics Days - JMDA16  in Barcelona, Spain on July 6-8, 2016, see \cite{rubio2016achromatic}.

Araujo-Pardo was partially supported by PAPIIT of Mexico grants IN107218, IN106318 and CONACyT of Mexico grant 282280.


\end{document}